\documentclass{amsart}

\newtheorem{thm}{Theorem}[section]

\newtheorem{lem}[thm]{Lemma}
\newtheorem{rem}[thm]{Remark}

\theoremstyle{definition}

\theoremstyle{remark}

\numberwithin{equation}{section}

\begin{document}

\title{A dichotomy result for prime algebras of Gelfand-Kirillov dimension two}


\author{Jason P. Bell}


\address{Department of Mathematics, Simon Fraser University, 
8888 University Drive, Burnaby, BC, Canada,
V5A 1S6}

\email{jpb@math.sfu.ca}


\thanks{The author thanks NSERC for its generous support.}

%    General info


\date{October 22, 2009.}


\begin{abstract}

Let $k$ be an uncountable field.  We show that a finitely generated prime Goldie $k$-algebra of quadratic growth is either primitive or satisfies a polynomial identity, answering a question of Small in the affirmative.  
\end{abstract}

\maketitle

\bibliographystyle{plain}

\section{Introduction}

We study finitely generated algebras in which we impose restrictions on the \emph{growth} of the algebras.  Given a field $k$ and a finitely generated $k$-algebra $A$, a $k$-subspace $V$ of $A$ is called a \emph{frame} of $A$ if $V$ is finite dimensional, $1\in V$, and $V$ generates $A$ as a $k$-algebra.

We say that $A$ has \emph{quadratic growth} if there exist a frame $V$ of $A$ and constants $C_1,C_2>0$ such that

$$C_1 n^2 \ \le \ {\rm dim}_k(V^n) \ \le \ C_2n^2\qquad {\rm for~all}~n\ge 1.$$

We note that an algebra of quadratic growth has Gelfand-Kirillov dimension $2$.  More generally, the \emph{Gelfand-Kirillov} dimension (GK dimension, for short) of a finitely generated $k$-algebra $A$ is defined to be
$${\rm GKdim}(A) \ = \ \limsup_{n\rightarrow\infty} \frac{\log\left({\rm dim}(V^n)\right)}{\log\, n},$$
where $V$ is a frame of $A$.  While algebras of quadratic growth have GK dimension $2$, it is not the case that an algebra of GK dimension $2$ necessarily has quadratic growth.  Constructions of algebras of GK dimension two that do not have quadratic growth tend to be contrived and are generally viewed as being pathological.  For instance, there are currently no examples of domains, simple rings, or prime noetherian rings of GK dimension $2$ that do not also have quadratic growth.  Indeed, Smoktunowicz \cite{Sm3} has shown that a graded domain whose GK dimension is in the interval $[2,3)$ must have quadratic growth.  For this reason, quadratic growth is viewed as being, for all intents and purposes, the same as GK dimension two for domains.

GK dimension can be viewed as a noncommutative analogue of Krull dimension in the following sense: if $A$ is a finitely generated commutative $k$-algebra then the Krull dimension of $A$ and the GK dimension of $A$ coincide.  Thus, the study of noncommutative finitely generated domains of quadratic growth can be viewed as the noncommutative analogue of the study of regular functions on affine surfaces.

Our main result is the following dichotomy theorem, which shows that a finitely generated prime Goldie algebra of quadratic growth is either very close to being commutative or it is primitive.  Given a field $k$, we say that a $k$-algebra $A$ satisfies a \emph{polynomial identity} if there is a nonzero noncommutative polynomial $p(t_1,\ldots ,t_d)\in k\{t_1,\ldots ,t_d\}$ such that $p(a_1,\ldots ,a_d)=0$ for all $(a_1,\ldots ,a_d)\in A^d$.  We note that a commutative ring satisfies the polynomial identity $t_1t_2-t_2t_1=0$.  In general, polynomial identity algebras behave very much like commutative algebras; in fact, a finitely generated prime $k$-algebra satisfying a polynomial identity always embeds in a matrix ring over a field.  Primitive algebras (i.e., algebras with a faithful simple left module), on the other hand, are very different from commutative algebras; indeed, a commutative algebra that is primitive is a field and a theorem of Kaplansky \cite[Theorem 13.3.8]{MR} generalizes this, showing that a primitive algebra that satisfies a polynomial identity is a matrix ring over a division algebra and, moreover, the division algebra is finite dimensional over its centre. 

There are many dichotomy results in the literature, which show that an algebra with certain specified properties is either primitive or satisfies a polynomial identity \cite{AS, BC, BS, FS, Ok}.  Occasionally, a trichotomy is proved in which one adds the possibility that the algebra may have a nonzero Jacobson radical.  Most of these dichotomies require severe restrictions on the algebra that make it easier to study.  Our dichotomy result for prime Goldie algebras is less restrictive than most of these other dichotomies, requiring only quadratic growth and an uncountable base field.

\begin{thm} \label{thm: main} Let $k$ be an uncountable field and let $A$ be a finitely generated prime Goldie $k$-algebra of quadratic growth.  Then either $A$ is primitive or $A$ satisfies a polynomial identity.  \end{thm}

In fact, we show that over any field $k$, if $A$ is a finitely generated prime Goldie $k$-algebra of quadratic growth, then either the set of prime ideals $P$ for which $A/P$ has GK dimension $1$ is finite or $A$ satisfies a polynomial identity.

The way this result is proved is by studying prime ideals $P$ in $A$ for which $A/P$ has GK dimension $1$.  We show there are only finitely many such primes unless $A$ satisfies a polynomial identity.  This result was proved by the author and Smoktunowicz \cite{BS} in the case that $A$ is a prime monomial algebra of quadratic growth using combinatorial techniques.  Here we use centralizers to obtain this result.  This intermediate result does not require an uncountable base field.  We then use an argument due to Farkas and Small \cite{FS} to show that if $A$ is a finitely generated prime Goldie algebra of GK dimension $2$ over an uncountable base field, and $A$ has only finitely many prime ideals $P$ for which $A/P$ has GK dimension $1$, then $A$ must be primitive.

The outline of the paper is as follows.  In Section 2, we give the proof of Theorem \ref{thm: main} and in Section 3 we give some remarks about our main result.

\section{Dichotomy}
In this section, we prove Theorem \ref{thm: main}.  To accomplish this, we work with the \emph{Goldie ring of quotients} of a prime Goldie algebra $A$ (see McConnell and Robson \cite[Chapter 2]{MR}).  

We recall that a ring $R$ is \emph{right Goldie} if it satisfies the following two conditions:
\begin{enumerate}
\item $R$ does not contain an infinite direct sum of nonzero right ideals;
\item $R$ satisfies the ascending chain condition on right annihilators.
\end{enumerate}

Left Goldie is defined analogously.  A ring that is both left and right Goldie is called Goldie.  A result of Irving and Small \cite{IS} shows that right Goldie and left Goldie are equivalent for semiprime algebras of finite GK dimension.  We note that a domain of finite GK dimension is Goldie \cite[Proposition 4.13]{KL}; also a prime noetherian algebra is Goldie.  

 A prime Goldie ring $A$ has a Goldie ring of quotients formed by inverting all nonzero elements; we denote this algebra by $Q(A)$---it is a matrix ring over a division ring.  In the case that $A$ is a Goldie domain, we call $Q(A)$ the \emph{quotient division algebra} of $A$.  We begin with a simple lemma that we will use to give an upper bound on how quickly a centralizer can grow in an algebra of quadratic growth.  
\begin{lem} 
Let $k$ be a field, let $A$ be a finitely generated $k$-algebra that is a Goldie domain, and let $E$ be a division subalgebra of $Q(A)$.  If $Q(A)$ is infinite-dimensional as a right $E$-vector space then
$AE$ is an infinite-dimensional right $E$-vector space.  In particular, if $V$ is a frame of $A$ then $V^nE$ is at least $n+1$-dimensional as a right $E$-vector space.\label{rem: 1}
\end{lem}
\begin{proof} Suppose that $AE$ has rank $m<\infty$ as an $E$-module.
 By assumption there exist $a_1,\ldots ,a_{m+1}\in Q(A)$ that are right-linearly independent over $E$.  There exists some nonzero $b\in A$ such that $a_i':=ba_i\in A$ for $1\le i\le m+1$.  Then by construction, $a_1'E+\cdots +a_{m+1}'E$ is direct, a contradiction. 
 
 Note that if $V$ is a frame for $A$ and 
rank $V^nE\le n$ for some $n$, then there must be some $i<n$ such that the rank of $V^iE$ is the same as the rank of $V^{i+1}E)$.  Thus $V^{i+1}E\subseteq V^iE$ and so by induction $V^mE\subseteq  V^iE$ for all $m\ge i$.  In particular, $AE\subseteq V^iE$, contradicting the fact that $AE$ is of infinite rank as a right $E$-module. 
\end{proof}
We now introduce centralizers into the proof. 
Given a ring $R$ and an element $r\in R$, we let $C(r;R)$ denote the 
\emph{centralizer} of $r$ in $R$; i.e., $$C(r;R)=\{x\in R~:~xr=rx\}.$$
\begin{lem}
Let $k$ be a field and let $A$ be a finitely generated $k$-algebra that is a domain of quadratic growth.  If $x\in A$ is not algebraic over $k$ and $V$ is a frame for $A$ containing $x$, then either there is a constant $C$ such that 
$${\rm dim}_k\left(V^n\cap C(x;A)\right) \le Cn$$ for all sufficiently large $n$ or $A$ satisfies a polynomial identity.
\label{lem: smok}
\end{lem}
\begin{proof} Assume that $A$ does not satisfy a polynomial identity and let $E=C(x;Q(A))$.  Then $E$ satisfies a polynomial identity by a theorem of Smoktunowicz \cite{Sm}.   Note that $Q(A)$ must be infinite-dimensional as a right $E$-vector space or else it too would satisfy a polynomial identity and thus so would $A$.  By Lemma \ref{rem: 1}, we have ${\rm dim}_E(V^nE)\ge n+1$.  Hence we can find $a_1,\ldots , a_{n+1}\in V^n$ that are right-linearly independent over $E$. 
Let $U_n$ be a basis for $V^n\cap C(x;A)$.    We claim that $\{a_iu~:1\le i\le n+1,~u\in U_n\}\subseteq V^{2n}$ is linearly independent over $k$.  To see this, note that if $\alpha_{i,u}\in k$ with $u\in U_n$, $1\le i \le n+1$ satisfy 
$$\sum_{u\in U_n}\sum_{i=1}^{n+1} \alpha_{i,u} a_iu \ = \ 0$$
then for each $i\le n+1$ we must have
$$\sum_{u\in U_n} \alpha_{i,u} u \ = \ 0$$ as $a_1,\ldots, a_{n+1}$ are right-linearly independent over $E\supseteq U_n$.  But by assumption, $U_n$ is linearly independent over $k$.  It follows that
$${\rm dim}(V^{2n})\ge (n+1)|U_n|.$$  The result now follows as the dimension of $V^n$ grows at most quadratically in $n$ and so there is a constant $C$ such that ${\rm dim}(V^{2n})\le Cn^2$ for all sufficiently large $n$.  Thus $|U_n|\le Cn$ for all sufficiently large $n$, and so we obtain the desired result. \end{proof}
We will use the preceding lemma along with growth estimates for algebras of GK dimension $1$ to obtain our main result.  We now prove some results for algebras of GK dimension $1$.  A result of Small and Warfield \cite{SW} shows that a finitely generated prime $k$-algebra $A$ of GK dimension $1$ satisfies a polynomial identity; in fact, they show that $A$ is a finite module over its centre, $Z(A)$.
\begin{lem} Let $k$ be a field, let $A$ be a finitely generated prime $k$-algebra of GK dimension $1$, and let $V$ be a frame for $A$.  If $Q(A)\cong M_d(F)$ for some finitely generated field extension $F$ of $k$ of transcendence degree $1$ then $$\liminf_{n\rightarrow\infty} \frac{1}{n}\cdot {\rm dim}(V^n\cap Z(A))\ge 1/d.$$\label{lem: d}
\end{lem}
\begin{proof} We identify $A$ with its image in $M_d(F)$.  By the Faith-Utumi theorem \cite[Theorem 3.2.6]{MR} there is some $c\in F$ and matrix units $e_{i,j}$ such that $ce_{i,j}\in A$ for $1\le i,j\le d$.  Pick $m$ such that
$ce_{i,j}\in V^m$ for $1\le i,j\le d$.  
Consider the subspace $ce_{1,1}V^n$.  As $A$ is infinite-dimensional over $k$, $ce_{1,1}V^n\subseteq V^{n+m}$ is at least $n+1$ dimensional over $k$.  Note that
\begin{eqnarray*}
n+1 &\le & ce_{1,1}V^nc \\
&=& ce_{1,1}V^n(ce_{1,1}+\cdots ce_{d,d})\\ 
&\subseteq & ce_{1,1}V^nce_{1,1} + \cdots + ce_{1,1}V^nce_{d,d}. \end{eqnarray*}
It follows that there is some $i\le d$ such that
$$ce_{1,1}V^n ce_{i,1}$$ is at least $p:=\lfloor (n+1)/d\rfloor$-dimensional over $k$.  Thus there exist
$b_1,\ldots ,b_p\in F$ that are linearly independent over $k$ such that $b_j e_{1,i}\in V^{n+2m}$ for $1\le j\le p$.  Thus
$$c^2b_j {\rm I}_d = \sum_{k=1}^d (ce_{k,1})b_je_{1,i}(ce_{i,k})\in V^{n+4m}$$ are $p$ linearly independent central elements in $V^{n+4m}$.  As $m$ is fixed, letting $n$ tend to infinity gives the desired result. 
\end{proof}
\begin{lem} \label{comm}  Let $k$ be a field, let $A$ be a finitely generated prime $k$-algebra of GK dimension $1$, and let $V$ be a frame of $A$.  Suppose that $Q(A)\cong M_d(F)$ for some finitely generated field extension $F$ of $k$ of transcendence degree $1$.  If $x\in V$ then
$${\rm dim}\left(V^n/(V^n\cap \{[a,x]~:~a\in Q(A)\})\right)>n/2$$ for all sufficiently large $n$.
\end{lem}
\begin{proof}
We identify $A$ with its image in $M_d(F)$.  Note that the operator $$[\, \cdot\, ,x]:M_d(F)\rightarrow M_d(F)$$ has kernel that is at least $d$-dimensional as an $F$-vector space.\footnote{This is well-known, but we do not know of a reference.  Let $V=F^d$.  Then $V$ is an $F[t]$-module, where $t$ acts by multiplication by $x$.  As $F[t]$ is a PID, $V$ decomposes as a direct sum of irreducible $F[t]$-modules: $V=V_1\oplus \cdots \oplus V_m$.  Then the centralizer of $x$ can be viewed as 
${\rm End}_{F[t]}(V,V)\supseteq \oplus_{i=1}^m {\rm End}_{F[t]}(V_i,V_i)$.    Then we note that ${\rm End}_{F[t]}(V_i,V_i)\ge {\rm dim}(V_i)$, as the powers of $t$ from $0$ to ${\rm dim}(V_i)-1$ give linearly independent endomorphisms.  From this we deduce that the centralizer is at least $d$-dimensional.}
Consequently, there exist $u_1,\ldots ,u_d$ in $M_d(F)$ that are linearly independent over $F$ such that the sum 
$$Fu_1+\cdots +Fu_d+\{[x,a]~:~a\in M_d(F)\}$$ is direct.  Using the Faith-Utumi theorem \cite[Theorem 3.2.6]{MR}, we may assume that $u_1,\ldots ,u_d\in V^p$ for some $p$.  We note that ${\rm dim}(V^n\cap Z(A))> 2n/3d$ for all $n$ sufficiently large by Lemma \ref{lem: d}.   We let $Z_n$ denote a basis for $V^n\cap Z(A)$.  Then the image of $$\bigcup_{i=1}^d Z_n u_i$$ in $${\rm dim}\left(V^{n+p}/(V^{n+p}\cap \{[a,x]~:~a\in Q(A)\})\right)$$ is linearly independent and has size at least $2n/3$ for all sufficiently large $n$.  As $p$ is fixed, we obtain the desired result. 
\end{proof}
We note that the preceding two lemmas have the strange hypothesis that the Goldie ring of quotients of $A$ is isomorphic to $M_d(F)$ with $F$ a field.  This does not happen in general, but there is an important case where this does occur.  
\begin{rem} Let $k$ be an algebraically closed field and let $A$ be a finitely generated prime $k$-algebra of GK dimension $1$.  Then $A\cong M_d(F)$ for some $d$ and some field $F$ which is a finitely generated field extension of $k$ of transcendence degree $1$.  \label{rem: tsen}
\end{rem}
\begin{proof}
We note that $Q(A)\cong M_d(D)$ for some division algebra that is finite dimensional over its centre \cite{SW}; moreover $Z(D)$ is a finitely generated field extension of $k$ of transcendence degree $1$.  By Tsen's theorem, the Brauer group of $Z(D)$ is trivial and so $D=Z(D)$. 
\end{proof}
We need one more remark before we prove our structure result for domains of quadratic growth.
  \begin{rem}\label{rem: cent} Let $k$ be a field and let $A$ be a prime finitely generated $k$-algebra of GK dimension $1$.  If $z\in Z(A), x\in A$ are nonzero and $U\subseteq A$, then the image of 
  $U$ in $A/A\cap \{[a,x]~:~a\in Q(A)\}$ is linearly independent over $k$, if and only if the image of $zU$ is linearly independent. \label{cent}
  \end{rem}
\begin{proof} Let $W=\{[a,x]~:~a\in Q(A)\}\cap A$.  Suppose that the image of $U$ is linearly dependent in $A/W$.  Then there exist $u_1,\ldots ,u_d\in U$ and $c_1,\ldots ,c_d\in k$, not all zero, such that 
$c_1u_1+\cdots +c_du_d\in W$.  Note that $W$ is closed under multiplication by $z$ as $[za,x]=z[a,x]$.  Thus $c_1(zu_1)+\cdots +c_d(zu_d)\in W$ and so $zU$ has linearly dependent image in $A/W$.  Conversely, suppose that
the image of $zU$ is linearly dependent in $A/W$.  Then $U =z^{-1}zU$ is linearly independent by the above argument.  (We note that nonzero central elements in a prime ring are necessarily regular.)
\end{proof}
 We use these remarks to prove the following result.

\begin{thm} Let $A$ be a finitely generated $k$-algebra that is a domain of quadratic growth.  Then either $A$ has only finitely many primes $P$ such that $A/P$ has GK dimension $1$ or $A$ satisfies a polynomial identity.\label{thm: finite}
\end{thm}
\begin{proof} Assume that $A$ does not satisfy a polynomial identity.  If $A$ is algebraic over $k$, it is a division algebra and we are done.  Thus we may assume there is $x\in A$ that is not algebraic over $k$.  Let $V$ be a frame for $A$ containing $x$.  By Lemma \ref{lem: smok} there is a constant $C$ such that ${\rm dim}(V^n\cap C(x;A))\le Cn$.  
We let $\overline{k}$ denote the algebraic closure of $k$ and we take
$$B=A\otimes_k \overline{k}.$$ The algebra $B$ need not be a domain, but $Q(A)\otimes_k \overline{k}$ is a localization of $B$.  We let $W=V\otimes_k \overline{k}\subseteq B$.  We identify $A$ with $A\otimes_k 1 \subseteq B$.    
Consider $$C_n:=\{[x,a]~:~a\in W^n\}\subseteq W^{n+1}.$$
Notice that $${\rm dim}(W^{n+1}/C_n)={\rm dim}(W^{n+1})-{\rm dim}(W^n) + {\rm dim}(W^n\cap C(x;B)).$$
Notice that ${\rm dim}(W^n\cap C(x;B))={\rm dim}(V^n\cap C(x;A))\le Cn$.
Pick $C_1>0$ such that ${\rm dim}(W^n)\le C_1n^2$ for all $n>1$.  Then a telescoping argument shows that
$${\rm dim}(W^{n+1})-{\rm dim}(W^n)\le 3C_1 n$$ for infinitely many $n$.  Consequently,
$${\rm dim}(W^{n+1}/C_n)\le (3C_1+C)n$$ for infinitely many $n$.  Pick an integer $d>6C_1+2C+2$.  Suppose $A$ has $d$ distinct primes $P_1,\ldots ,P_d$ with ${\rm GKdim}(A/P_i)=1$ for $1\le i\le d$.  Notice that $P_i$ lifts to an ideal $I_i$ of $B$.  Moreover, since $B/I_i$ is infinite-dimensional, there is a prime ideal $Q_i$ of $B$ such that $B/Q_i$ has GK dimension $1$ and $A\cap Q_i=P_i$.  Moreover, $Q_1,\ldots ,Q_d$ are distinct as $Q_i\cap A=P_i$.  Let $B_i=B/Q_i$ and let $T_i\subseteq B_i$ denote the set $B_i\cap\{[b,\overline{x}]~:~b\in Q(B_i)\}$, where $\overline{x}$ denotes the image of $x$ in $B_i$.  
Then we can find elements $z_i\in \cap_{j\neq i} Q_j$ such that the image of $z_i$ in $B_i$ is nonzero and central.
Pick $p$ such that $z_i\in W^p$ for $1\le i\le d$.  By Remark \ref{rem: tsen}, the Goldie ring of quotients of $B_i$ is isomorphic to a matrix ring over a field $F$ that is a finitely generated extension of $\overline{k}$ of transcendence degree $1$.  
By Lemma \ref{comm} there exists a positive integer $N$ such that for all $n>N$ and $i\le i\le d$, we can find
a $k$-linearly independent set $U_{i,n}\subseteq W^{n+1}$ whose size is greater than $n/2$ and has the property that the image of $U_{i,n}$ in the vector space $B_i/T_i$ is linearly independent.

We claim that the set $$\{z_i u+C_{n+p}~:~1\le i\le d,~u\in U_{i,n}\}\subseteq W^{n+1+p}/C_{n+p},$$ which has size at least $dn/2>(3C_1+C+1)n$, is linearly independent.
If not, there exist scalars $c_{i,u}\in \overline{k}$, not all equal to zero, such that 
$$\sum_{i=1}^d \sum_{u\in U_{i,n}} c_{i,u} z_i u \in C_{n+p}.$$
If we consider this sum mod $Q_j$, using the fact that $z_i\in Q_j$ for $i\neq j$, we see that
$$\sum_{u\in U_{j,n}} c_{j,u} z_j u \in C_{n+p}+Q_j.$$  In fact, we chose the set $U_{j,n}$ to have the property that its image in $B_j/T_j$ is linearly independent. Consequently, the image of $z_jU_j$ is linearly independent in $B_j/T_j$ by Remark \ref{cent} and so
we see $c_{j,u}=0$ for all $u\in U_{j,n}$.  Thus 
$${\rm dim} \left(W^{n+1+p}/C_{n+p}\right) > (3C_1+C+1)n.$$  On the other hand, we showed that
$${\rm dim}(W^{n+1+p}/C_{n+p})\le (3C_1+C)(n+p)$$ for infinitely many $n$, a contradiction, as $p$ is fixed.  The result follows. 
\end{proof}
As a corollary of this theorem, we obtain our main result.  The missing ingredients are an argument of Farkas and Small \cite{FS} and the following remark.
\begin{lem} Let $k$ be an uncountable field and let $A$ be a finitely generated prime Goldie $k$-algebra of GK dimension two that does not satisfy a polynomial identity.  Then $A$ has at most countably many height one prime ideals $Q$ with the property that $A/Q$ is finite-dimensional as a $k$-vector space. \label{lem: countable}
\end{lem}
\begin{proof}  Suppose not.  Then there exist an integer $d$ and an uncountable set of height one primes $S$ such that $A/Q$ is $d$-dimensional for each $Q$ in $S$.  

Let $I$ denote the intersection of the primes in $S$.  If $I=(0)$ then $A$ embeds in a direct product of $d$-dimensional rings and hence satisfies a polynomial identity \cite[Corollary 13.1.13]{MR}.  Notice also that $I$ has infinite codimension by the Chinese remainder theorem.  Thus $A/I$ must have GK dimension $1$ \cite[Proposition 3.15]{KL}.  Since a finitely generated semiprime ring of GK dimension $1$ is noetherian \cite{SSW}, there are only finitely many minimal primes in $A/I$.  But by construction, each $Q\in S$ is a minimal prime above $I$, a contradiction.  
\end{proof}
\begin{proof}[Proof of Theorem \ref{thm: main}.]
We prove our main theorem first in the case that $A$ is a domain.
Assume $A$ is not PI.  Then by Theorem \ref{thm: finite} there are only finitely many primes $P$ such that $A/P$ has GK dimension $1$.  Let $P_1,\ldots ,P_d$ denote these primes.  By Lemma \ref{lem: countable} $A$ has only countably many primes $Q$ for which $A/Q$ is finite-dimensional and such that $Q$ does not contain any of $P_1,\ldots ,P_d$.
   
Since $A$ is finitely generated and $k$ is uncountable, the Jacobson radical is a nil ideal and hence is $(0)$ as $A$ is a domain.   It follows that there is non-algebraic $x$ in the intersection of the primes $P_1,\ldots ,P_d$.  Let $T$ denote the set of all prime ideals $Q$ such that $A/Q$ is finite-dimensional and such that $Q$ does not contain $P_1\cap\cdots \cap P_d$.  Then $T$ is countable.  For each $Q$ in $T$, the algebra $A/Q$ is finite-dimensional and hence there are at most finitely many $\lambda\in k$ for which the image of $x-\lambda$ in $A/Q$ is not a unit in $A/Q$.  Since $T$ is countable there are at most countably many $\lambda\in k$ for which the image of $x-\lambda$ in $A/Q$ is not a unit in $A/Q$ for some $Q\in T$.  Since $k$ is uncountable, we can pick nonzero $\alpha\in k$ such that the image of $x-\alpha$ in $A/Q$ is a unit in $A/Q$ for every $Q\in T$.
Thus
$$A(x-\alpha)+Q=A$$ for every $Q\in T$.  By construction $A(x-\alpha)+P_i=A$ for $1\le i\le d$.  Let $L$ be a maximal left ideal containing $A(x-\alpha)$.  Then by construction $A/L$ is a simple module and the annihilator is $(0)$.  Thus $A$ is primitive.

In the case that $A$ is prime Goldie, we note that $Q(A)\cong M_d(D)$ for some natural number $d$ and some division algebra $D$.  We regard $A$ as a subalgebra of $M_d(D)$ and identify $D$ with the scalar matrices in $M_d(D)$.  We note that if $D$ has the property that every finitely generated subalgebra satisfies a polynomial identity, then $M_d(D)$ also has this property; since $A$ is a finitely generated subalgebra, it would then satisfy a polynomial identity.  Thus we may assume that $D$ is not locally PI; i.e., there exists a finitely generated subalgebra of $D$ that does not satisfy a polynomial identity.  Let $B=D\cap A$.  By the Faith-Utumi theorem \cite[Theorem 3.2.6]{MR}, $Q(B)=D$.  

By the above remarks, we can find a finitely generated subalgebra $C$ of $B$ with the property that $C$ does not satisfy a polynomial identity.  We claim that every nonzero prime ideal $P$ of $A$ intersects $C$ non-trivially.  If not, $C$ embeds in $A/I$.  But $A/P$ has GK dimension at most $1$ and hence satisfies a polynomial identity, contradicting the fact that $C$ does not satisfy a polynomial identity.  

By Theorem \ref{thm: finite}, $C$ has only finitely many prime ideals $Q$ such that $C/Q$ has GK dimension $1$.  Let $Q_1,\ldots ,Q_m$ denote these primes. Pick a nonzero element $x\in Q_1\cap \cdots \cap Q_m$.  By Lemma \ref{lem: countable} there are 
%only countably many height one primes $Q$ of $A$ such that $A/Q$ is finite-dimensional and 
only countably many height one primes $P$ of $C$ such that $C/P$ is finite-dimensional.  Using the argument of Farkas and Small again, we see that for each height one prime $P$ of $C$ of finite codimension, there are only finitely many $\lambda\in k$ such that $C(1-\lambda x)+Q\neq C$.  %Similarly, for each height one prime $P$ of $C$ of finite codimension, there are only finitely many $\lambda\in k$ such that $(1-\lambda x)C+P\neq C$.  
Thus we can pick a nonzero $\lambda$ such that 
%$A(1-\lambda x)+Q=A$
%for every height one prime $Q$ of $A$ of finite codimension and such that
$C(1-\lambda x)+P=C$ for every height one prime $P$ of $C$ of finite codimension.  

We note that if $P$ is a prime ideal of $C$ such that $C/P$ has GK dimension $1$, then $x$ is in $P$ and hence $$C= P+C(1-\lambda x),$$ and so we see that $C(1-\lambda x)+P=C$ for every nonzero prime ideal of $A$. 

We claim that $A(1-\lambda x)+P=A$ for every nonzero prime $P$ of $A$.  To see this note, that $C/(P\cap C)$ embeds in $A/P$ and hence $C/(P\cap C)$ satisfies a polynomial identity.  Thus there is a finite set of prime ideals $P_1,\ldots ,P_r$ of $C$ that are minimal above $P\cap C$ \cite[Corollary 13.4.4]{MR}.  By construction $C(1-\lambda x)+P_i = C$ for $1\le i\le r$.   Since $P_1\cap\cdots \cap P_r$ is nilpotent over $C\cap P$ \cite{Br}, we see that $C(1-\lambda x)+(P\cap C) = C$. It follows that $A(1-\lambda x)+P=A$.  Thus $A$ is primitive using the argument above.
\end{proof}
 \section{Concluding remarks}
 In this section, we make a few remarks about our main result.
 
Artin and Stafford \cite{AS} showed in the course of their description of graded domains of GK dimension two that all graded domains of GK dimension $2$ that are generated in degree $1$ are either primitive or satisfy a polynomial identity.  This result follows from their classification of these algebras and uses geometric techniques.  We do not give a classification of domains of quadratic growth as Artin and Stafford do in the graded case, but we do prove their dichotomy result in the ungraded case.  In fact, Theorem \ref{thm: finite} has interesting geometric repercussions: it shows that an automorphism of a complex curve either has finite order or has at most finitely many periodic points.  This is of course well-known, but what is interesting is that our proof of this fact is purely combinatorial.  We also note that the reduction argument used in the proof of Theorem \ref{thm: main} can be used to show that a prime graded Goldie algebra of quadratic growth is primitive or satisfies a polynomial identity.  Here we note that in the non-PI case, height one primes are necessarily graded and so there is at most one height one prime ideal of finite codimension; namely, the homogeneous maximal ideal.  Thus the uncountable field hypothesis is not needed in the graded case.

The author and Smoktunowicz \cite{BS} constructed a finitely generated prime algebra $A$ of GK dimension $2$ that does not satisfy a polynomial identity and has infinitely many primes $P$ such that $A/P$ has GK dimension $1$.  We note, however, this algebra does not have quadratic growth.  The author \cite{Be2} has also constructed examples of prime rings of GK dimension $2$ (but again not of quadratic growth) that do not satisfy the ascending chain condition on prime ideals.  Thus without some prime Goldie hypothesis, one cannot expect the conclusion of the statement of Theorem \ref{thm: finite} to hold.

We also note that for prime algebras of GK dimension two there are counterexamples \cite{Be1, SV} which show that the conclusion of the statement of Theorem \ref{thm: main} does not hold.  In both of these examples, the algebra has a locally nilpotent Jacobson radical.

If one follows the arguments in this paper carefully, then if $A$ is a finitely generated domain of quadratic growth that does not satisfy a polynomial identity, one can give an upper bound on the number of prime ideals $P$ for which $A/P$ has GK dimension $1$ in terms of the growth of a frame $V$ of $A$.  More specifically, there is a function $F:(0,\infty)\rightarrow (0,\infty)$ such that if $A$ does not satisfy a polynomial identity, the number of such primes is at most $F(C)$, where the constant $C$ is chosen to satisfy 
$${\rm dim}(V^n)\le Cn^2$$ for all $n$ sufficiently large. It would be interesting to find the best upper bound possible.

\section*{ Acknowledgments }
I thank Toby Stafford, Dan Rogalski, Agata Smoktunowicz, and James Zhang, with whom I have discussed this problem many times over the years and who have all made interesting observations about this problem.  Most of all, I thank Lance Small, who gave me this problem now nearly nine years ago while I was doing my PhD.  I thank him for his many interesting observations and also for encouraging me to continue working on this problem, telling me I was close to a solution after I had decided the problem was hopeless.  Sure enough, he was right.
\bibliographystyle{amsplain}

\end{document}